\documentclass{amsart}

\usepackage{amsmath, amsthm, amssymb, amsfonts}
\usepackage{scrpage2}
\usepackage{exscale}
\usepackage{pstricks}

\theoremstyle{plain}

\newtheorem{theorem}{Theorem}[section]

\newtheorem{lemma}[theorem]{Lemma}

\newtheorem{corollary}[theorem]{Corollary}

\theoremstyle{definition}

\newtheorem{definition}[theorem]{Definition}

\newtheorem{assumption}[theorem]{Assumption}

\theoremstyle{remark}
\newtheorem{remark}[theorem]{Remark}

\newtheorem{example}[theorem]{Example}
\newtheorem*{proofmain}{Proof of the main theorem}
\newtheorem*{proofhigher}{Proof of the Theorem \ref{theoremhigher}}

\newtheorem*{ak}{Acknowledgements}

\DeclareSymbolFont{AMSb}{U}{msb}{m}{n}
\DeclareMathSymbol{\N}{\mathalpha}{AMSb}{"4E}
\DeclareMathSymbol{\R}{\mathalpha}{AMSb}{"52}
\DeclareMathSymbol{\Z}{\mathalpha}{AMSb}{"5A}
\DeclareMathSymbol{\D}{\mathalpha}{AMSb}{"44}
\DeclareMathSymbol{\s}{\mathalpha}{AMSb}{"53}

\newcommand{\sX}{\scriptscriptstyle{X}}
\newcommand{\sM}{\scriptscriptstyle{M}}
\newcommand{\sN}{\scriptscriptstyle{N}}
\newcommand{\sOmega}{\scriptscriptstyle{\Omega}}
\newcommand{\sK}{\scriptscriptstyle{K}}

\newcommand{\X}{X}
\newcommand{\dx}{\de_{\sX}}
\newcommand{\sV}{\scriptscriptstyle{V}}

\DeclareMathOperator{\Ch}{Ch}
\DeclareMathOperator{\lip}{Lip}

\DeclareMathOperator{\spec}{spec}

\DeclareMathOperator{\supp}{supp}

\DeclareMathOperator{\de}{d}
\DeclareMathOperator{\m}{m}
\DeclareMathOperator{\ric}{ric}
\DeclareMathOperator{\diam}{diam}

\DeclareMathOperator{\Ent}{Ent}

\newcommand{\ChX}{\Ch^{\sX}}
\providecommand{\keywords}[1]{\textbf{\textit{Index terms---}} #1}

\title{Obata's rigidity theorem for metric measure spaces}
\author{Christian Ketterer}
\email{christian.ketterer@math.uni-freiburg.de}

\begin{document}\maketitle
\begin{abstract}
We prove Obata's rigidity theorem for metric measure spaces that satisfy a Riemannian curvature-dimension condition. Additionally, we show that a lower bound $K$ for the generalized Hessian of a sufficiently regular function $u$ holds if and only if $u$ is $K$-convex. A corollary is also a rigidity result for higher order eigenvalues.
\end{abstract}
\noindent
\\
\tableofcontents
\section{Introduction}
In this article we prove Obata's eigenvalue rigidity theorem in the context of metric measure spaces satisfying a Riemannian curvature-dimension condition. More precisely, our main result is 
\begin{theorem}\label{AAA}
Let $(X,\de_{\sX},\m_{\sX})$ be a metric measure space that satisfies the condition $RCD^*(K,N)$ for $K\geq 0$ and $N\geq 1$. If $N>1$, we assume $K>0$, and if $N=1$ we assume $K=0$ and $\diam_{\sX}\leq \pi$.
There is $u\in D^2(L^{\sX})$ such that
\begin{itemize}
 \item[(i)] $L^{\sX}u=-\textstyle{\frac{K N}{N-1}}u \ \mbox{ if }\ N>1,$
 \smallskip
 \item[(ii)] $L^{\sX}u=-u \ \mbox{ otherwise }.$
\end{itemize}
Then, $\diam_{\sX}=\pi\sqrt{\frac{N-1}{K}}$ if $N>1$, and $\diam_{\sX}=\pi$ if $N=1$.
\end{theorem}
\noindent
$D^2(L^{\sX})$ is the domain of the generalized Laplace operator $L^{\sX}$.
A consequence of this result and the maximal diameter theorem \cite{ketterer2} is the following rigidity result.
\begin{theorem}\label{brbrbr}
Let $(X,\de_{\sX},\m_{\sX})$ be a metric measure space that satisfies $RCD^*(N-1,N)$ for $N\geq 1$, and let $u$ such that $L^{\sX}u=-Nu$.
Then, there exists a metric measure space $(X',\de_{\sX'},\m_{\sX'})$ such that $(X,\de_{\sX},\m_{\sX})\simeq[0,\pi]\times_{\sin}^{\sN-1} X'$ and 
\smallskip
\begin{itemize}
\item[(1)] $X'$ satisfies $RCD^*(N-2,N-1)$ and $\diam_{X'}\leq \pi$ if $N\geq 2$, or 
\smallskip
\item[(2)] $X'=\left\{x_0\right\}$ and $\m_{\sX'}=c\cdot\delta_{x_0}$ for some constant $c>0$ if $N\in (1,2)$, or 
\smallskip
\item[(3)] $X'$ as in (2), or $X'=\left\{N,S\right\}$ with $\de_{\sX}(N,S)=\pi$ and $\m_{\sX'}=c\cdot(\delta_{S}+\delta_{N})$ for some constant $c>0$, if $N=1$.
\end{itemize}
\end{theorem}
\noindent
Our proof of Theorem \ref{AAA} relies on the self-improvement property of the Bakry-Emery condition (Theorem \ref{crucial}) and a gradient comparison result for eigenfunctions with sharp eigenvalue (Theorem \ref{wichtigeaussage}) 
that is an application of a maximum principle for sub-harmonic functions on general metric measure spaces. Additionally, we will prove in section \ref{hes} that a lower bound $K$ for the generalized Hessian of a sufficiently regular Sobolev function
holds if and only if the function is $K$-convex (Theorem \ref{convexity}). This result may be of independent interest.
\begin{corollary}\label{brbr}
Let $(X,\de_{\sX},\m_{\sX})$ be a metric measure space that satisfies the condition $RCD^*(K,N)$ for $K\geq 0$ and $N\geq 1$ where $K>0$ if $N>1$, and $K=0$ and $\diam_{\sX}\leq \pi$ if $N=1$. 
Then following statements are equivalent.
\begin{itemize}
\smallskip
\item[(1)] $\diam_{\sX}=\pi\sqrt{{\textstyle\frac{N-1}{K}}}$,
\smallskip
\item[(2)] $\inf \left\{\spec_{\sX}\backslash \left\{0\right\}\right\}={\frac{KN}{N-1}}$,
\smallskip
\item[(3)] $X\!=\![0,\pi]\times_{\sin_{K/(N-1)}}^{N-1}X'$ for a metric measure space $X'$, and $u=c\cdot\cos_{{\scriptscriptstyle {K}/{N-1}}}$.
\end{itemize}
\end{corollary}
\noindent
The corollary is a new link between the Lagrangian picture of Ricci curvature that comes from optimal transport and the Eulerian picture that is encoded via properties
of the energy and the corresponding generalized Laplacian. In Section \ref{higher} we prove a similar rigidity result for the higher eigenvalues.
\begin{theorem}\label{theoremhigher}
Let $(X,\de_{\sX},\m_{\sX})$ be a metric measure space that satisfies the condition $RCD^*(N-1,N)$ for $N\geq 1$, and $\diam_{\sX}\leq \pi$ if $N=1$. 
Assume $\lambda_k=N$ for $k\in\mathbb{N}$.
If $k\leq N$, then
there exists a metric measure space $Z$ such that
$$X=\mathbb{S}_+^k\times ^{N-k}_{\sin \circ \de_{\partial \mathbb{S}_+^k}}Z.$$
$\mathbb{S}_+^k$ denotes the upper, closed hemisphere of the $k$-dimensional standard sphere $\mathbb{S}^k$. Additionally,
\begin{itemize}
\smallskip
 \item[(1)] If $N-k\geq 1$, $Z$ satisfies $RCD^*(N-k-1,N-k)$, and
 \smallskip
 \item[(2)] If $0< N-k<1$, then $Z$ consists of one point, or, 
 \smallskip
 \item[(3)] If $N=k$, $Z$ consists of exactly one point or two points with distance $\pi$ and the measure is given as in Theorem \ref{brbrbr}.
 \smallskip
 \end{itemize}
In particular, if $k=N$, $X=\mathbb{S}^N$ or $X=\mathbb{S}_+^{\sN}$ where $c\cdot\m_{\sX}=$Riemannian volume for some constant $c$. In the former case one also gets $\lambda_{k+1}=N$.
If $k>N$, then $k=N+1$ and $X=\mathbb{S}^{\sN}$ where the reference measure is again the Riemannian volume times a constant.
\end{theorem}
\noindent
We remark that a metric measure space $(X,\de_{\sX},\m_{\sX})$ that satisfies a curvature-dimension $CD(K,N)$ \cite{lottvillanilichnerowicz} or a reduced Riemannian curvature dimension condition $RCD^*(K,N)$
for some positive constant $K\in (0,\infty)$ and $N\in (1,\infty)$ satisfies the sharp Lichnerowicz spectral gap inequality:
\begin{align}\label{lichner}
\int_X (f-\bar{f}_{\sX})^2d\m_{\sX}\leq {\textstyle \frac{N-1}{KN}}\int_X (\mbox{Lip}f)^2d\m_{\sX}
\end{align}
where $f$ is a Lipschitz function, $\bar{f}_{\sX}
$ its mean value with respect to $\m_{\sX}$ and $\mbox{Lip}f(x)=\limsup_{y\rightarrow x}\frac{|f(x)-f(y)|}{\de_{\sX}(x,y)}$ the local Lipschitz constant.
In a setting where we have Riemannian curvature-dimension bounds, the local Lipschitz constant yields a quadratic energy functional and this estimate describes the spectral gap of the corresponding self-adjoint operator $L^{\sX}$.
It is also a simple consequence of Bochner's inequality that was established for general $RCD^*$-spaces by Erbar, Kuwada and Sturm in \cite{erbarkuwadasturm}.
For Riemannian manifolds with positive lower Ricci curvature bounds the theorem was proven by Obata in \cite{obata}.
\begin{theorem}[Obata, 1962]
Let $(M,g_{\sM})$ be a $n$-dimensional Riemannian manifold with $\ric_{\sM}\geq K>0$. Then, there exists a function $u$ such that $$\Delta u=-\textstyle{\frac{Kn}{n-1}}u$$ 
if and only if $M$ is the standard sphere ${\scriptscriptstyle \sqrt{\frac{K}{n-1}}}\cdot S^n$.
\end{theorem}
\noindent
Since a Riemannian manifold $M$ satisfies the condition $RCD^*(K,N)$ if and only if $\dim_{\sM}\leq N$ and $\ric_{\sM}\geq K$, and
since a spherical suspension that is a manifold without boundary, is a sphere, our result also covers Obata's theorem. 
In the context of Alexandrov spaces with curvature bounded from below an Obata type theorem was proven by Qian, Zhang and Zhu \cite{qzz}. They use a different notion of generalized lower Ricci curvature bound that implies a sharp curvature
dimension condition and is inspired by
Petrunin's second variation formula \cite{petruninsecondvariation}.

In the next section we briefly introduce important definitions. In section 3 we compute the Hessian of an eigenfunction $u$ and obtain further properties. In section 4 we 
will derive a gradient comparison result (Theorem \ref{wichtigeaussage}), and finally in section 5 we will prove the main Theorem \ref{AAA}. In section 6 we prove the result that bounds for the Hessian of a sufficiently smooth functions are equivalent 
to metric semi-convexity (Theorem \ref{convexity}), and we use it to obtain more detailed properties of the eigenfunction $u$ (Corollary \ref{bf}). This yields Corollary \ref{brbr}.  Then, in section 7 we also obtain higher eigenvalue rigidity results, and in section 8 we will give a brief outlook to 
the non-Riemannian situation.
\begin{ak}
I would like to thank Yu Kitabeppu for his interest in the results of this work and many fuitful discussions.  I also want to thank Luigi Ambrosio for his interest and for reading carefully an early version of the article.
\end{ak}

\section{Preliminaries}
\subsection{Curvature-dimension condition}
Let $(X,\de_{\sX})$ be a complete and separable metric space equipped with a locally finite Borel measure $\m_{\sX}$. 
The triple $(X,\de_{\sX},\m_{\sX})$ will be called \textit{metric measure space}. 
\smallskip

$(X,\de_{\sX})$ is called \textit{length space} 
if $\de_{\sX}(x,y)=\inf \mbox{L} (\gamma)$ for all $x,y\in X$, 
where the infimum runs over all absolutely continuous curves $\gamma$ in $X$ connecting $x$ and $y$ and $\mbox{L}(\gamma)$ is the length of $\gamma$. 
$(X,\de_{\sX})$ is called \textit{geodesic space} 
if every two points $x,y\in X$ are connected by a curve $\gamma$ such that $\de_{\sX}(x,y)=\mbox{L}(\gamma)$.
Distance minimizing curves of constant speed are called \textit{geodesics}. 
A length space, which is complete and locally compact, is a geodesic space.
$(X,\de_{\sX})$ is called \textit{non-branching} 
if for every quadruple $(z,x_0,x_1,x_2)$ of points in $X$ for which $z$ is a midpoint of $x_0$ and $x_1$ as well as of $x_0$ and $x_2$, it follows that $x_1=x_2$.
\smallskip

We say a function $f:(X,\de_{\sX})\rightarrow \mathbb{R}\cup\left\{\infty\right\}$ is \textit{$Kf$-concave} if for any geodesic $\gamma:[0,1]\rightarrow X$ the composition $f\circ\gamma$ satisfies 
$
u''+K\theta^2u\leq 0
$
in the distributional sense where $\theta=\mbox{L}(\gamma)$. We say that $f$ is \textit{weakly $Kf$-concave} if there exists a geodesic such 
that the previous differential inequality holds. We say $f$ is $Kf$-convex (weakly $Kf$-convex) if $-f$ is $Kf$-concave (weakly $Kf$-concave). In the same way
we define $K$-convexity (concavity) and weak convexity (concavity). If $f:X\rightarrow \mathbb{R}$ is convex and concave, we say it is affine.
\smallskip

$\mathcal{P}_2(X)$ denotes the \textit{$L^2$-Wasserstein space} of probability measures $\mu$ on $(\X,\de_{\sX})$ with finite second moments.
The \textit{$L^2$-Wasserstein distance} $\de_{W}(\mu_0,\mu_1)$ between two probability measures
$\mu_0,\mu_1\in\mathcal{P}_2(\X,\dx)$ is defined as
\begin{equation}\label{wassersteindist}
\de_{W}(\mu_0,\mu_1)^2={\inf_{\pi}\int_{\X\times \X}\dx^2(x,y)\,d\pi(x,y)
}.
\end{equation}
Here the infimum ranges over all \textit{couplings} of $\mu_0$ and $\mu_1$, 
i.e. over all probability measures on $\X\times \X$ with marginals $\mu_0$ and $\mu_1$. $(\mathcal{P}_2(X),\de_{W})$ is a 
complete separable metric space. The subspace of $\m_{\sX}$-absolutely continuous probability measures is denoted by $\mathcal{P}_2(\X,\m_{\sX})$.
A minimizer of (\ref{wassersteindist}) always exists and is called \textit{optimal coupling} between $\mu_0$ and $\mu_1$. 
\smallskip

%

\begin{definition}[Reduced curvature-dimension condition, \cite{bast}]\label{cdcondition}
Let $(X,\de_{\sX},\m_{\sX})$ be a metric measure space. 
It satisfies the condition $CD^*(K,N)$ for $K\in\mathbb{R}$ and $N\in[1,\infty)$ if for each pair
$\mu_0,\mu_1\in\mathcal{P}_2(X,\m_{\sX})$ there exists an optimal
coupling $q$ of $\mu_0=\rho_0\m_{\sX}$ and $\mu_1=\rho_1\m_{\sX}$ and a geodesic 
$\mu_t=\rho_t \m_{\sX}$ in $\mathcal{P}_2(X,\m_{\sX})$ connecting them such that
\begin{align*}
\int_{X}\rho_t^{-{\scriptscriptstyle \frac{1}{N'}}}\rho_td\m_{\sX}
\ge\int_{X\times
X}\!\big[\sigma^{(1-t)}_{K,N'}(\de_{\sX})\rho^{-\frac{1}{N'}}_0(x_0)+
\sigma^{(t)}_{K,N'}(\de_{\sX})\rho^{-\frac{1}{N'}}_1(x_1)\big]dq(x_0,x_1)
\end{align*}
for all $t\in (0,1)$ and all $N'\geq N$ where $\de_{\sX}:=\de_{\sX}(x_0,x_1)$.
In the case $K>0$, the \textit{volume distortion coefficients} $\sigma^{(t)}_{K,N}(\cdot)$
for  $t\in (0,1)$  are defined by
\begin{align*}
\sigma_{K,N}^{(t)}(\theta)=\textstyle{\frac {\sin_{K/N}\left(\theta t\right)} {\sin_{K/N}\left(\theta\right)} }
\end{align*}
if $0\le\theta< \scriptstyle{\sqrt{\frac{N}K}\pi}$ and by $\sigma_{K,N}^{(t)}(\theta)=\infty$ if $K\theta^2\geq{N}\pi^2$.
The generalized $\sin$-functions $\sin_{K}$ are defined by
$$
\sin_{K}(t)=\sin(\sqrt{\textstyle{K}}t) \ \ \ \mbox{ for } \ K\in \mathbb{R}.
$$
The generalized $\cos$-functions are $\cos_{K}=\sin_{K}'$.
If $K\theta^2=0$, one sets $\sigma_{0,N}^{(t)}(\theta)=t$, and in the case $K<0$ one has to 
replace $\sin\scriptscriptstyle{\left(\sqrt{\frac K{N}}-\right)}$ by $\sinh\scriptscriptstyle{\left(\sqrt{\frac{-K}{N}}-\right)}$. 
\end{definition}
\subsection{Riemannian curvature-dimension condition}
We introduce some notations for the calculus on metric measure spaces. For more details see for instance \cite{agsriemannian,agsheat,agsbakryemery}. Let $L^2(\m_{\sX})=L^2(X)$ be the Lesbegue-space of $(X,\de_{\sX},\m_{\sX})$.
For a function $u:X\rightarrow \mathbb{R}\cup \left\{\pm\infty\right\}$ the local Lipschitz constant is denoted by $\mbox{Lip}:X\rightarrow [0,\infty]$.
It is finite, if $u\in \mbox{Lip}(X)$ - the space of Lipschitz continuous functions on $(X,\de_{\sX})$.
For  $u\in L^2(\m_{\sX})$ the \textit{Cheeger energy} is defined by 
\begin{align}\label{liggett}
 2\ChX(u)=\liminf_{\underset{v \underset{\scriptscriptstyle{L^2}}{\longrightarrow} u}{v\in \lip(X)}}\int_{\sX}\left(\lip v\right)^2d\m_{\sX}.
\end{align}
If $\ChX(u)<\infty$, we say $u\in D(\ChX)$. We also use the notation $D(\ChX)=W^{1,2}(X)=W^{1,2}(\m_{\sX})$. If $\ChX(u)<\infty$, then
\begin{align}
2\ChX(u)=\int_{\sX}|\nabla u|_w^2d\m_{\sX}
\end{align}
where $|\nabla u|_w\in L^2(\m_{\sX}\!)$ is the \textit{minimal weak upper gradient} of $u$. The definition can be found in \cite{agsriemannian}. 
$W^{1,2}(X)$ equipped with the norm 
$
\left\|u\right\|_{W^{1,2}}^2=\left\|u\right\|_{L^2}^2+2\ChX(u)
$
becomes a Banach space. Since the the minimal weak upper gradient is local, there is also a well-defined space of local Sobolov functions $W^{1,2}_{loc}(X)$.
The Cheeger energy is convex and lower semi-continuous but it is not a quadratic form in general. A metric measure space is said to be infinetisimally Hilbertian if $\ChX$ is quadratic.
\begin{definition}[\cite{erbarkuwadasturm, agsriemannian, giglistructure}]
A metric measure space 
$(X,\de_{\sX},\m_{\sX})$ 
satisfies the (reduced) Riemannian curvature-dimension condition $RCD^*(K,N)$ if 
$(X,\de_{\sX},\m_{\sX})$ is infinitesimally Hilbertian and 
satisfies the condition $CD^*(K,N)$.
\end{definition}
\begin{remark}
If $X$ satisfies a Riemannian curvature-dimension condition, $\ChX$ is a strongly local and regular Dirichlet form and the set of Lipschitz functions in $W^{1,2}(X)$ is dense in 
$W^{1,2}(X)$ with respect to $\left\|\ \cdot \ \right\|_{W^{1,2}}$.
There is a bilinear and symmetric map $\langle\nabla\cdot,\nabla\cdot\rangle:W^{1,2}(X)\times W^{1,2}(X)\rightarrow L^{1}(\m_{\sX})$
where $\langle\nabla u,\nabla u\rangle=|\nabla u|^2$. For the rest of the article we assume that $X$ is infinitesimally Hilbertian.
\end{remark}
\begin{theorem}[Generalized Bonnet-Myers Theorem]\label{bonnet}
Assume that a metric measure space  $(X,\de_{\sX},\m_{\sX})$ satisfies $RCD^*(K,N)$ for some $K>0$ and $N> 1$. 
Then the diameter of $(X,\de_{\sX})$ is bounded by $\pi\scriptstyle{\sqrt{\frac{N-1}{K}}}$.
\end{theorem}
\subsection{Bakry-Emery condition}
There is a self-adjoint, negative-definite operator $\left(L^{\sX}, D_2(L^{\sX})\right)$ on $L^2(\m_{\sX})$ that is defined via integration by parts.
Its domain is 
$$D_2(L^{\sX})=\left\{u\in D(\mathcal{E}^{\sX})\!:\exists v\in L^2(\m_{\sX}):-(v,w)_{L^2(\m_{\sX})}=\mathcal{E}^{\sX}(u,w)\ \forall w\in D(\mathcal{E}^{\sX})\right\}.$$ 
We set $v=:L^{\sX}u$.
$D_2(L^{\sX})$ is dense in $L^2(\m_{\sX})$ and equipped with the topology given by the graph norm.
$L^{\sX}$ induces a strongly continuous Markov semi-group $(P^{\sX}_t)_{t\geq 0}$ on $L^2(X,m_{\sX})$.
The correspondence between form, operator and semi-group is standard (see \cite{fukushima}). 
\smallskip
\\
The $\Gamma_2$-operator is defined by 
\begin{align*}
2\Gamma^{\sX}_2(u,v;\phi)=\int_X\langle\nabla u,\nabla v\rangle L^{\sX}\phi d\m_{\sX}-\int_X\left[\langle \nabla u,\nabla L^{\sX}v\rangle +\langle\nabla v,\nabla L^{\sX}u\rangle\right]\phi d\m_{\sX}
\end{align*}
for $u,v\in D(\Gamma^{\sX}_2)\mbox{ and }\phi\in D^{\infty}_+(L^{\sX})$ where
$$
D(\Gamma^{\sX}_2):=\left\{u\in D(L^{\sX}): L^{\sX}u\in W^{1,2}(X)\right\}
$$
and 
$$
D^{\infty}_+(L^{\sX}):=\big\{\phi\in D^2(L^{\sX}): \phi, L^{\sX}\phi\in L^{\infty}(\m_{\sX}), \phi\geq 0\big\}.
$$
We set $\Gamma^{\sX}_2(u,u;\phi)=\Gamma^{\sX}_2(u;\phi)$.
\begin{definition}[Bakry-Emery curvature-dimension condition]\label{bakrycd2}
Let $K\in\mathbb{R}$ and $N\in [1,\infty]$. We say that $\ChX$ satisfies the Bakry-Emery curvature-dimension condition $BE(K,N)$ if 
for every $u\in D(\Gamma^{\sX}_2)\mbox{ and }\phi\in D^{\infty}_+(L^{\sX})$
\begin{align}\label{crucialestimate}
\Gamma^{\sX}_2(u;\phi)\geq K \int_X|\nabla u|^2\phi d\m_{\sX}+\frac{1}{N}\int_{\sX}(L^{\sX}u)^2\phi d\m_{\sX}.
\end{align}
The implications $BE(K,N)\Rightarrow BE(K,N')\Rightarrow BE(K,\infty)$ for $N'\geq N$ hold.
\end{definition}
\begin{assumption}\label{TheAss}
$(X=\supp{\m_{\sX}},\de_{\sX},\m_{\sX})$ is a geodesic metric measure space. 
Every $u\in W^{1,2}(X)$ with $|\nabla u|_w\leq 1$ a.e. 
admits a $1$-Lipschitz version $\tilde{u}$. This regularity assumption is always satisfied for $RCD^*$-spaces.
\end{assumption}
\noindent

\begin{theorem}[\cite{agsbakryemery, erbarkuwadasturm}]\label{theorembochner}Let $K\in \mathbb{R}$ and $N\in [1,\infty]$.
Let $(X,\de_{\sX},\m_{\sX})$ be a metric measure space that satisfies the condition $RCD^*(K,N)$. Then
\begin{itemize}
 \item[(1)] $BE(K,N)$ holds for $\ChX$.
\end{itemize}
Moreover, if $(X,\de_{\sX},\m_{\sX})$ is a metric measure space that is infinitesimal Hilbertian, satisfies the Assumption \ref{TheAss} and $\ChX$ satisfies the condition $BE(K,N)$,
then 
\begin{itemize}
 \item[(2)] $(X,\de_{\sX},\m_{\sX})$ satisfies $CD^*(K,N)$, i.e.
the condition $RCD^*(K,N)$.
\end{itemize}
\end{theorem}
\begin{remark}
In the setting of $RCD(K,N)$-spaces with finite $N$, it is know 
that the semi-group $P^{\sX}_t$ is $L^2$-$L^{\infty}$-ultra-contractive (see \cite[Proposition 4]{agsriemannian}).
\end{remark}
\begin{theorem}\label{dist}
Let $X$ be a metric measure space that satisfies $RCD^*(K,N)$ for $K\in \mathbb{R}$ and $N>1$. Then it follows for all $x,y\in X$ that
\begin{align*}
\de_{\sX}(x,y)=\sup\left\{u(x)-u(y): \ u\in W^{1,2}_{loc}(X) \mbox{ such that }|\nabla u|\leq 1\ \m_{\sX}\mbox{-a.e.}\right\}.
\end{align*}

\end{theorem}

\section{Hessian identity}
For the rest of the article let $(X,\de_{\sX},\m_{\sX})$ be metric measure space that satisfies the Riemannian curvature-dimension condition $RCD(K,N)$ for $K>0$ and $N\geq 1$.
\begin{definition}[The space $\mathbb{D}^{\sX}_{\infty}$]
We introduce another function space.
\begin{align*}
\mathbb{D}^{\sX}_{\infty}&=\left\{f\in D(L^{\sX})\cap L^{\infty}(\m_{\sX}):|\nabla f|\in L^{\infty}(\m_{\sX})\ \ \& \ \ L^{\sX}f\in W^{1,2}(X)\right\}
\end{align*}
By definition $\mathbb{D}^{\sX}_{\infty}\subset D(\Gamma_2^{\sX})$. 
In particular, $f$ is Lipschitz continuous. Savar\'e proved in \cite{savareself} the following implication:
\begin{align*}
u\in \mathbb{D}^{\sX}_{\infty}\ \ \Longrightarrow \ \  |\nabla u|^2\in W^{1,2}(X)\cap L^{\infty}(\m_{\sX})
\end{align*}
The regularization properties of $P^{\sX}_t$ also imply that $P^{\sX}_tL^2(\m_{\sX})\subset \mathbb{D}^{\sX}_{\infty}$. Hence, $\mathbb{D}_{\infty}^{\sX}$ is dense in $W^{1,2}(\m_{\sX})$ and $D^2(L^{\sX})$ 
(provided $X$ satisfies $RCD^*(K,N)$ for finite $N$). We follow the notation from \cite{savareself}.
\end{definition}
In \cite{giglistructure} Gigli studied the properties of Sobolev functions $f\in W^{1,2}_{loc}(X)$ that admit a measure-valued Laplacian. 
We briefly present his approach. We assume the space $X$ is compact. In this case Gigli's definition simplifies. 
For the general definition we refer to Definition 4.4 in \cite{giglistructure}.
\begin{definition}[\cite{giglistructure}]\label{measuredlap}
Let $\Omega$ be an open subset in $X$. We say that $u\in W^{1,2}_{loc}(X)$ is in the domain of the Laplacian in $\Omega$, $u\in D(\textbf{L},\Omega)$ if 
there exists a signed Radon measure $\mu=:\textbf{L}^{\sOmega}u$ on $\Omega$ such that for any 
Lipschitz function $f$ with $\supp f\subset \Omega'$ and $\Omega'\subset \Omega$ open, it holds
\begin{align*}
-\int_{\Omega}\langle \nabla f,\nabla u\rangle d\m_{\sX}=\int_{\Omega}fd\mu.
\end{align*}
The set of such test functions $f$ is denoted by $\mbox{Test}(\Omega)$.
There is also the following integration by parts formula (Lemma 4.26 in \cite{giglistructure}). If $u\in D(\textbf{L},\Omega)$ such that $\textbf{L}^{\sOmega}u=hd\m_{\sX}$ for $h\in L^2_{loc}(\m_{\sX}|_{\Omega})$, then for every $v\in W^{1,2}(X)$ with support in $\Omega$ it holds
\begin{align*}
-2\ChX(u,v)=-\int_{\Omega}\langle \nabla u,\nabla v\rangle d\m_{\sX}=\int_{\Omega}vhd\m_{\sX}.
\end{align*}
Hence, if $\Omega=X$, in this case we have $\textbf{L}^{\sX}u=L^{\sX}u$ and the definition is consistent with the previous definition that comes from Dirichlet form theory. Additionally, the definition is also consistent 
with the notion of measure-valued Laplacian that is 
used by Savar\'e in \cite{savareself} (at least if we assume the space is compact) where he uses the notation $\mathbb{M}_{\infty}^{\sX}=D(\textbf{L},X)$.
\end{definition}
\begin{remark}
A direct consequence of Definition \ref{measuredlap} is the following ``global-to-local'' property.
\begin{align*}
\Omega'\subset \Omega \mbox{ open }\&\ u\in D(\textbf{L},\Omega)\ \Longrightarrow \ u|_{\Omega'}\in D(\textbf{L},\Omega')\ \& \ \left(\textbf{L}^{\sOmega}u\right)|_{\Omega'}=\textbf{L}^{\sOmega'}(u|_{\Omega'}).
\end{align*}
\end{remark}
\begin{lemma}[Lemma 3.4 in \cite{savareself}]\label{sav}
If $\ChX$ satisfies $BE(\kappa,\infty)$ then for every $u\in \mathbb{D}_{\infty}^{\sX}$ we have $|\nabla u|^2\in \mathbb{M}_{\infty}^{\sX}$ with 
\begin{align*}
\frac{1}{2}\textbf{L}^{\sX}|\nabla u|^2-\langle \nabla u,\nabla L^{\sX}u\rangle\m_{\sX} \geq \kappa |\nabla u|^2 \m_{\sX}
\end{align*}
Moreover, $\mathbb{D}_{\infty}^{\sX}$ is an algebra (closed w.r.t. pointwise multiplication) and 
if $\textbf{f}=(f_i)_{i=1}^n\in \left[\mathbb{D}_{\infty}^{\sX}\right]^{n}$ then $\Phi(\textbf{f}\hspace{1pt})\in \mathbb{D}_{\infty}^{\sX}$ for every smooth function $\Phi:\mathbb{R}^n\rightarrow \mathbb{R}$ with $\Phi(0)=0$.
\end{lemma}
\noindent
As consequence of the previous lemma Savar\'e introduced the measure-valued $\Gamma_2^{\sX}$ for any $u\in \mathbb{D}_{\infty}^{\sX}$.
\begin{align*}
{\boldsymbol\Gamma}_2^{\sX}(u):=\frac{1}{2}{\boldsymbol L}^{\sX}|\nabla u|^2-\langle \nabla u,\nabla L^{\sX}u\rangle\m_{\sX}
\end{align*}
is a finite Borel measure that has finite total variance. ${\boldsymbol\Gamma}_2^{\sX}(u)$ can be Lebesgue decomposed with respect to $\m_{\sX}$, and we denote by $\gamma_2^{\sX}\in L^{1}(\m_{\sX})$ its density with respect to $\m_{\sX}$. If we follow Savar\'e in \cite{savareself}, we see that
\begin{align}\label{gamma2}
{\boldsymbol\Gamma}_2^{\sX}(u)\geq\gamma_2^{\sX}(u)\geq \kappa |\nabla u|^2 +\frac{1}{N}\left(L^{\sX}u\right)^2
\end{align}
provided the condition $RCD^*(K,N)$ holds (the singular part is non-negative).
\begin{definition}\label{hessianfirst}
For $u\in \mathbb{D}^{\sX}_{\infty}$ we can define the Hessian. That is a bilinear, symmetric operator on $\mathbb{D}^{\sX}_{\infty}$.
\begin{align*}
H[u]:\mathbb{D}^{\sX}_{\infty}\times \mathbb{D}^{\sX}_{\infty}&\rightarrow L^1(\m_{\sX})\\
H[u](f,g)&=\frac{1}{2}\left(\langle \nabla f,\nabla\langle\nabla u,\nabla g\rangle\rangle+\langle \nabla g,\nabla\langle\nabla u,\nabla f\rangle\rangle-\langle \nabla u,\nabla\langle\nabla f,\nabla g\rangle\rangle\right).
\end{align*}
\end{definition}
\begin{lemma}[\cite{savareself}]
Let $\textbf{f}=(f_i)_{i=1}^{n}\in \left(\mathbb{D}_{\infty}^{\sX}\right)^n$ and let $\Phi\in C^3(\mathbb{R})$ with $\Phi(0)=0$. Then
\begin{align}\label{key}
\gamma_2^{\sX}(\Phi(\textbf{f}\hspace{1pt}))=&\sum_{i,j}\Phi_i(\textbf{f}\hspace{1pt})\Phi_j(\textbf{f}\hspace{1pt})\gamma_2^{\sX}(f_i,f_j)\nonumber\\
&+2\sum_{i,j,k}\Phi_i(\textbf{f}\hspace{1pt})\Phi_{j,k}(\textbf{f}\hspace{1pt})H[{f_i}](f_j,f_k)\nonumber\\
&+\sum_{i,j,k,h}\Phi_{i,k}(\textbf{f}\hspace{1pt})\Phi_{j,h}(\textbf{f}\hspace{1pt})\langle\nabla f_i,\nabla f_j\rangle\langle\nabla f_k,\nabla f_h\rangle.
\end{align}
\end{lemma}
%
%
%
\begin{theorem}\label{crucial}
Let $(X,\de_{\sX},\m_{\sX})$ satisfy $RCD^*(K,N)$ for $K>0$ and $N> 1$.
Consider $u\in D(L^{\sX})$ such that $$L^{\sX}u=-\frac{K N}{N-1}u.$$
Then $u\in \mathbb{D}^{\sX}_{\infty}$ and $H[u](f,g)=-\textstyle{\frac{K}{N-1}} u\langle\nabla f,\nabla g\rangle$ $\m_{\sX}$-a.e. for any $f,g\in \mathbb{D}^{\sX}_{\infty}$, and $|\nabla u|^2+\frac{K}{N-1}u^2=const$ $\m_{\sX}$-a.e.\ .
\end{theorem}
\begin{proof}
The Bakry-Ledoux gradient estimate yiels immediately that  $u\in \mathbb{D}^{\sX}_{\infty}$. Hence $H[u](f,g)$ is well-defined for all $f,g\in \mathbb{D}^{\sX}_{\infty}$.
It follows that $|\nabla u|^2\in \mathbb{M}_{\infty}$ and $\gamma_2^{\sX}(u)$ exists.
Now, since there is the $\Gamma_2$-estimate (\ref{gamma2}) and since there is the statement of the previous lemma, we can perform exactly the same calculations as in the proof of 
Theorem 3.4 in \cite{savareself} to obtain a self-improved, sharp $\Gamma_2$-estimate that also involves terms depending on the dimension. 
These are precisely the calculations that Sturm did in \cite{sturmconformal} where the existence of a nice functional algebra is assumed. In our setting the result of the previous lemma is sufficient to do the same calculations. We get
\begin{align*}
\Big[H[u](g,f)-{\textstyle\frac{1}{N}}\langle\nabla g,\nabla f\rangle L^{\sX}u\Big]^2\leq \frac{C}{2}\left[\gamma^{\sX}_2(u)-\textstyle{\frac{1}{N}}\big(L^{\sX}u\big)^2d-K|\nabla u|^2\right]
\end{align*}
where $
C:=\frac{2(N-1)}{N}|\nabla g|^2|\nabla f|^2>0\ \m_{\sX}\mbox{-a.e.}
$
if $N>1$ and $|\nabla f|,|\nabla g|>0$.
Integration with respect to $\m_{\sX}$ yields
\begin{align*}
&\int_X\frac{2}{C}\left[H[u](g,f)+\textstyle{\frac{K}{N-1}}\langle\nabla g,\nabla f\rangle u\right]^2d\m_{\sX}\\
&\hspace{0.9cm}\leq \int\gamma^{\sX}_2(u) d\m_{\sX}-K\int_X|\nabla u|^2d\m_{\sX}-\frac{1}{N}\int_X\left(L^{\sX}u\right)^2d\m_{\sX}\\
&\hspace{0.9cm}=0+\int_X\left(L^{\sX}u\right)^2d\m_{\sX}+K\int_XuL^{\sX}ud\m_{\sX}-\frac{1}{N}\int_X\left(L^{\sX}u\right)^2d\m_{\sX}\\
&\hspace{0.9cm}=\left[\textstyle{\frac{K^2N^2}{(N-1)^2}-\frac{K^2N}{N-1}-\frac{K^2N}{(N-1)^2}}\right]\int_X u^2d\m_{\sX}=0.
\end{align*}
Hence, we obtain the first result.
\smallskip
\\
Since $u\in \mathbb{D}_{\infty}^{\sX}$, we know that $\textstyle{\frac{K}{N-1}}u^2+|\nabla u|^2\in W^{1,2}(X)$. We observe that 
$$H[u](u,g)=\frac{1}{2}\langle \nabla g,\langle\nabla u,\nabla u\rangle\rangle=-\textstyle{\frac{K}{N-1}}u\langle\nabla u,\nabla g\rangle.$$ Hence
\begin{align*}
\langle \nabla (\textstyle{\frac{K}{N-1}}u^2+|\nabla u|^2),\nabla g\rangle= \textstyle{\frac{K}{N-1}}2u\langle\nabla u,\nabla g\rangle + \langle\nabla|\nabla u|^2,\nabla g\rangle=0
\end{align*}
for any $g\in \mathbb{D}^{\sX}_{\infty}$. It implies $\textstyle{\frac{K}{N-1}}u^2+|\nabla u|^2$ is constant $\m_{\sX}$-a.e.\ .
\end{proof}

\section{A gradient comparison result}
\begin{example}
We introduce $1$-dimensional model spaces. For $K>0$, $N\geq 1$ and some interval $[a,b]\subset I_{K/(N-1)}$ let us consider $\cos_{K/(N-1)}:[a,b]\rightarrow \mathbb{R}_{\geq 0}$ where
\begin{align*}
I_{K/(N-1)}=\textstyle{\left[-\frac{\pi}{2}{\scriptscriptstyle \sqrt{\frac{{(N-1)}}{{K}}}},\frac{\pi}{2}{\scriptscriptstyle \sqrt{\frac{{(N-1)}}{{K}}}}\right].}
\end{align*}
The metric measure space $([a,b], \m_{\sK,\sN})$ with $d\m_{\sK,\sN}=\cos^{\sN-1}_{\sK/(\sN-1)}rdr$ satisfies the condition $RCD^*(K,N)$ and for $u\in C^{\infty}([a,b])$ with Neumann boundary condition the generalized Laplacian is given by
\begin{align*}
L^{I_{K/(N-1)}}u=\frac{d^2}{dr^2}u-N\frac{\sin_{{\scriptscriptstyle K/(N-1)}}}{\cos_{{\scriptscriptstyle K/(N-1)}}}\frac{d}{dr}u.
\end{align*}
If $[a,b]=I_{K/(N-1)}$, one can check that $-\sin_{K/(N-1)}:I_{K/(N-1)}\rightarrow \mathbb{R}$ is an eigenfunction for the eigenvalue $\frac{KN}{N-1}$. If $u$ is an eigenfunction of $([a,b], \m_{\sK,\sN})$
for the eigenvalue $\lambda$, then the classical Lichnerowicz estimate tells $\lambda_1\geq \frac{KN}{N-1}$. In the following we often write $v=-c\cdot\sin_{K/(N-1)}$ for some $c=const>0$.
\end{example}
\begin{remark}\label{remark}
By the previous theorem we have $|\nabla u|^2=const-\textstyle{\frac{K}{N-1}}u^2$ $\m_{\sX}$-almost everywhere. It follows that actually $|\nabla u|^2\in D^2(L^{\sX})$ and
\begin{align*}
L^{\sX}|\nabla u|^2=-\textstyle{\frac{K}{N-1}}L^{\sX}u^2=-\textstyle{\frac{2K}{N-1}}uL^{\sX}u-2|\nabla u|^2= \textstyle{\frac{2K^2N}{(N-1)^2}}u^2-2|\nabla u|^2 \in W^{1,2}(X).
\end{align*}
By ultra-contractivity of $P^{\sX}_t$, by the Bakry-Ledoux gradient estimate and since $u$ is an eigenfunction we know that 
$$|\nabla |\nabla u|^2|^2=|\nabla (\textstyle{\frac{K}{N-1}}u^2)|^2=\big(\textstyle{\frac{2K}{N-1}}\big)^2u^2|\nabla u|^2\leq CP_t^{\sX}|\nabla u|^2\in L^{\infty}(\m_{\sX})$$
for some constant $C=C(t)>0$.
This yields
\begin{align*}
|\nabla L^{\sX}|\nabla u|^2|\leq C\left(u|\nabla u|+|\nabla |\nabla u|^2|\right)\in L^{\infty}(\m_{\sX}).
\end{align*}
and it follows that $L^{\sX}|\nabla u|^2$ is Lipschitz continuous by the regularity properties of $RCD$-spaces (see Assumption \ref{TheAss}).
\end{remark}
The next theorem is the main result of \cite{giglimondino} (see also \cite{giglistructure}). It shows that Definition \ref{measuredlap} is compatible with local minimizers of the Cheeger energy.
For what follows we assume that $X$ satisfies a local $2$-Poincar\'e inequality and has a doubling property. These properties are fullfilled if $X$ satisfies a curvature-dimension condition. $\Omega$ is an open subset of $X$.
We say that $u$ is a sub-minimizer of $\ChX$ on $\Omega$ if 
$$\int_{\Omega}|\nabla u|^2d\m_{\sX}\leq \int_{\Omega}|\nabla(u+f)|^2d\m_{\sX}\ \ \mbox{ for all non-positive }f\in \mbox{Test}(X).$$
\begin{theorem}[Theorem 4.3 in \cite{giglimondino}]\label{min} Let $u\in D({\boldsymbol L}^{},\Omega)$. Then the following statements are equivalent: 
\begin{itemize}
 \item[(i)] $u$ is sub-harmonic: $$-\int_{\Omega}\langle\nabla u,\nabla f\rangle d\m_{\sX}\leq 0 \mbox{ for all non-positive }f\in \mbox{Test}(X)$$
 \item[(ii)] $u$ is a sub-minimizer of $\ChX$ on $\Omega$
\end{itemize}
\end{theorem}
A very important consequence of this characterization is the strong maximum principle that was established for sub-minimizers by Bj\"orn/Bj\"orn in \cite{bjoern}.
\begin{theorem}[Strong maximum principle]\label{max}
Let $u$ be a sub-minimizer of $\ChX$ in $\Omega$, and $\Omega$ has compact closure. If $u$ attains its maximum in $\Omega$, then $u$ is constant.
\end{theorem}
\noindent 
The maximum principle is the main ingridient in the proof of the following theorem.
\begin{theorem}\label{wichtigeaussage}
Let $(X,\de_{\sX},\m_{\sX})$ satisfy $RCD^*(K,N)$ for $N>1$.
Consider $u\in D(L^{\sX})$ as in Theorem \ref{crucial}.
Then
\begin{align*}
\langle\nabla u,\nabla u\rangle\leq (v'\circ v^{-1})^2(u).
\end{align*} 
where $v:I_{\sK/(\sN-1)}\rightarrow \mathbb{R}$ is an eigenfunction for the eigenvalue $\lambda_1=\textstyle{\frac{KN}{N-1}}$ with Neumann boundary conditions of the $1$-dimensional model space 
$(I_{\sK/(\sN-1)},\m_{\sK,\sN})$
such that $[\min u,\max u]\subset [\min v,\max v]$.
\end{theorem}
\begin{proof}
We follow ideas of Kr\"oger \cite{kroeger} and Bakry/Qian \cite{bakryqian}.
If we consider $\Psi\in C^{\infty}(\mathbb{R}^2)$ with bounded first and second derivatives, and $u_1,u_2\in \mathbb{D}_{\infty}^{\sX}$, then we have $\Psi(u_1,u_2)\in D(L^{\sX})$ and we can compute $L^{\sX}\Psi(u_1,u_2)$ explicetly. 
More precisely, we have
\begin{align*}
L^{\sX}\Psi(u_1,u_2)=\sum_{i=1}^{2}\Psi_i(u_1,u_2)L^{\sX}u_i+\sum_{i,j=1}^{2}\Psi_{i,j}(u_1,u_2)\langle \nabla u_i,\nabla u_j\rangle.
\end{align*}
One can actually check that $\Psi(u,v)\in \mathbb{D}^{\sX}_{\infty}$. In particular, $\Psi(u,v)\in \mathbb{D}^{\sX}_{\infty}$ 
is Lipschitz continuous. 
\smallskip\\
We set $u_1=u$, $u_2=|\nabla u|^2$ and $\Psi(u_1,u_2)=\psi(u_1)(u_2-\phi(u_1))$ for non-negative auxiliary functions $\psi,\phi\in C^{\infty}([\min u,\max u])$ with bounded first and second derivatives. 
We obtain
\begin{align}\label{plug}
L^{\sX}\Psi(u,|\nabla u|^2)=&\left[\psi'(u)|\nabla u|^2-\psi'(u)\phi(u)-\psi(u)\phi(u)\right]L^{\sX}u \nonumber\medskip\nonumber\\
&+\psi(u)L^{\sX}|\nabla u|^2 +2\psi'(u)\langle \nabla u,\nabla |\nabla u|^2 \rangle \medskip\nonumber\\
&+\left[\psi''(u)|\nabla u|^2-\psi''(u)\phi(u)-2\psi'(u)\phi'(u)-\psi(u)\phi''(u)\right]|\nabla u|^2.
\end{align}
We apply the result of Theorem \ref{crucial}:
\begin{align}
\langle\nabla u,\nabla |\nabla u|^2\rangle&= -2\lambda_1\frac{u}{N}|\nabla u|^2.\label{22}
\end{align}
Since $u$ is an eigenfunction, we can see from Remark \ref{remark} that $L^{\sX}\Psi(u,|\nabla u|^2)\in \mathbb{D}_{\infty}^{\sX}$. In particular, $L^{\sX}\Psi(u,|\nabla u|^2)$ is Lipschitz-continuous.
%
\smallskip
\smallskip\\
\textit{Claim:} If $\Psi(u,|\nabla u|^2)$ attains its maximum in $p\in X$, then $0\geq L^{\sX}\Psi(u,|\nabla u|^2)(p)$.
\smallskip\\
\textit{Proof of the claim:} We set $\Psi(u,|\nabla u|^2)=f$. The claim follows from the maximum principle (see Theorem \ref{max} above).
First, the Laplacian $L^{\sX}f$ coincides with the notion of measure-valued Laplacian from Definition \ref{measuredlap}. 
Therefore, $L^{\sX}f$ can be localized to an open subset $\Omega$.
Assume $L^{\sX}f(p)>0$. Then by Lipschitz continuouity of $L^{\sX}f$ there exists an open ball $B_{\delta}(p)$ such that $(L^{\sX}f)|_{B_{\delta}(p)}\geq c>0$.
By the localization property we can say that $f|_{B_{\delta}(p)}$ is sub-harmonic, that is 
\begin{align*}
0\geq \int_{B_{\delta}(p)}L^{\sX}f g d\m_{\sX}=-\int_{B_{\delta}(p)}\langle \nabla f,\nabla g\rangle d\m_{\sX}
\end{align*}$\mbox{for any non-positive }g\in \mbox{Test}(X).$
By Theorem \ref{min} $f|_{B_{\delta}(p)}$ is a sub-minimizer of the Cheeger energy in $B_{\delta}(p)$. Finally, for these minimizers the strong maximum principle holds (Theorem \ref{max}). 
Hence, $f|_{B_{\delta}(p)}$ is constant. But this contradicts $(L^{\sX}f)|_{B_{\delta}(p)}>0$. \qed
\smallskip
\\
In the next step we consider
\begin{align}\label{11}
L|\nabla u|^2\ \geq \ \ & 2K|\nabla u|^2 +\frac{2}{N}\left(L^{\sX}u\right)^2 + 2\langle \nabla u,\nabla L^{\sX}u\rangle\nonumber\\
\ =\ \ &2K |\nabla u|^2+\frac{2}{N}\lambda_1^2u^2-2\lambda_1|\nabla u|^2
\end{align}
and plug (\ref{11}) also into (\ref{plug}).
We set $F(u)=\Psi(u,|\nabla u|^2)=\psi(u)(|\nabla u|^2-\phi(u))$, $|\nabla u|^2=F(u)/\psi(u)+\phi(u)$,
and $\phi=(v'\circ (v)^{-1})>0$.
We choose $v$ such that $[\min u,\max u]\subset (\min v,\max v)$. Therefore, we have that $\phi\in C^{\infty}([\min u,\max u])$. 
A straightforward computation yields
\begin{align}\label{kiki}
L^{\sX}F(u)\geq &F^2\left(\frac{\psi''}{\psi^2}\right)+F\left(\frac{\psi''}{\psi}\phi - \frac{\psi'}{\psi}\left(4\lambda u \textstyle{\frac{1}{N}}+2\phi'+\lambda u\right)-\phi''-2\lambda +2K\right)\nonumber\\
&+\psi\left(\phi'\lambda u+2 K\phi - 2\lambda \phi +\textstyle{\frac{2}{N}\lambda^2}u^2-\phi''\phi\right)-\psi'\left(4\lambda u\frac{1}{N}\phi+2\phi\phi'\right).
\end{align}
where $F=F(u)$, $\psi=\psi(u)$ and $\phi=\phi(u)$ etc. In particular, the computation holds for $u=v$ and any admissible $\psi$. In this case we have $F=0$ and equality in (\ref{kiki}).
Consequently, the last line of the previous equation is $0$.
Additionally, for any $x\in X$ we can choose $\psi$
such that $\psi(u(x))=\epsilon$ for some arbitrarily small $\epsilon>0$ and $\psi'(u(x))=c>0$. Therefore, it follows that $2\lambda x\frac{1}{N}=\phi'(x)$ for any $x\in [\min u,\max u]$. Hence, we obtain
\begin{align}\label{prev}
L^{\sX}F(u)\geq &\underbrace{\frac{1}{\psi(u)}\left[\frac{\psi''(u)}{\psi(u)}\right]}_{=:g(\psi)}F^2(u)+\underbrace{\left[\frac{\psi''(u)}{\psi(u)}\phi(u)-\frac{\psi'(u)}{\psi(u)}{\lambda}u\right]}_{=:h(\psi)}F(u).
\end{align}
If there is a positive $\psi\in C^{\infty}(\mathbb{R})$ such that $g,h>0$ on $[\min u,\max u]$, we can conclude that $F\leq 0$. Otherwise $F(p)>0$ and (\ref{prev}) implies that $L^{\sX}F(u)(p)>0$. But
this contradicts the previous claim.
\medskip
\\
Consider $H(t)=\psi(v(t))$ and compute its second derivative
\begin{align*}
H''(t)=\psi''(v(t))(v'(t))^2+\psi'(v(t))v''(t)=\psi''(v(t))\phi(v(t))-\psi'(v(t))\lambda_1v(t).
\end{align*}
Therefore, if we choose $H=\cosh$, we see that $\psi(u)=\cosh(v^{-1}(u))>0$ solves $h(\psi)=1>0$. We also compute that 
\begin{align*}
\psi''(u)=\frac{1}{(v'\circ v^{-1})^2}\left[\cosh(v^{-1}(u))+\sinh(v^{-1}(u))\frac{u\circ v^{-1}}{v'\circ v^{-1}}\right]
\end{align*}
on $[\min u,\max u]$. Hence, $g(\psi)>0$ on $[\min u,\max u]$.
Finally, we obtain the statement for general $v$ if we replace $v$ by $cv$ for $c>1$. Then let $c\rightarrow 1$.
\end{proof}
\begin{corollary}
The statement of Theorem \ref{AAA} holds provided $\min u=-\max u$.
\end{corollary}
\begin{proof}Assume $N>1$. 
The previous theorem implies that 
\begin{align*}
|\nabla v^{-1}\circ u|^2\leq 1
\end{align*}
where $v=\min u\cdot\sin_{K/(N-1)}:I_{K/(N-1)}\rightarrow \mathbb{R}$. Then, Theorem \ref{dist} yields $\diam_{\sX}= {\pi}\textstyle{\sqrt{{(N-1)}/{K}}}$. Hence, $X$ attains the maximal diameter and we can
use the maximal diameter theorem from \cite{ketterer2}.

If $N=1$, we can argue as follows. The curvature-dimension condition implies that the Hausdorff-dimension is $1$. Even more, the space consists of two points that have maximal distance and at most $2$ geodesics that 
connect them. Hence, $X$ isomorphic to $c\cdot \mathbb{S}^1$ for some $0\leq c\leq 1$. But the existence of $u$ such that $L^{\sX}u=-u$ forces $c$ to be equal $1$.
\end{proof}
\begin{corollary}\label{auchwichtig}
Let $(X,\de_{\sX},\m_{\sX})$ satisfy $RCD^*(K,N)$ for $N>1$.
Consider $u\in D(L^{\sX})$ as in Theorem \ref{crucial}.
Then
\begin{align*}
\langle\nabla u,\nabla u\rangle\leq (w'\circ w^{-1})^2(u).
\end{align*} 
where $w:[a,b]\rightarrow \mathbb{R}$ is an eigenfunction for the eigenvalue $\lambda_1=\scriptstyle{{KN}/{(N-1)}}$ of the generalized Laplace operator $L^{\sK',\sN'}$ with Neumann boundary conditions for some $1$-dimensional model space 
$([a,b],\m_{\sK',\sN'})$ with $K'\leq K$ and $N'\geq N$
such that $[\min u,\max u]\subset [\min w,\max w]$.
\end{corollary}
\begin{proof} We observe that $(I_{\sK,\sN},\m_{\sK,\sN})$ satisfies $RCD^*(K',N')$. 
Bakry/Qian pove in \cite{bakryqian} (Theorem 8) that 
$$
(v')^2\leq(w'\circ w^{-1})^2(v)
$$
where $v$ is an eigenfunction of $(I_{\sK,\sN},\m_{\sK,\sN})$ for the eigenvalue $\lambda_1=\textstyle{\frac{KN}{N-1}}$ such that $[\min v,\max v]=[\min w,\max w]$.
\end{proof}
\section{Proof of the main theorem}
As we already mentioned the statement of the main theorem would already be true if we had $\min u=-\max u$. In the last step we establish this identity. Again, we will apply ideas Bakry/Emery \cite{bakryqian},
and the situation simplifies significantly since $\lambda_1=\textstyle{\frac{KN}{N-1}}$.
First, we have the following Theorem.
\begin{theorem}
Let $X$ and $u$ be as in Theorem \ref{AAA}. Consider the model space $([a,b],\m_{\sK',\sN'})$ for some interval $[a,b]$ and a eigenfunction $w$ for the eigenvalue $\lambda_1={\textstyle\frac{K N}{N-1}}$ where $K'<K$ and $N'> N$.
If $[\min u,\max u]\subset [\min w,\max w]$,
then 
\begin{align*}
R(c)=\left({\int_{\left\{u\leq c\right\}}ud\m_{\sX}}\right)\left({\int_{\left\{v\leq c\right\}}vd\m_{\sK',\sN'}}\right)^{-1}
\end{align*}
is non-decreasing on $[\min u,0]$ and non-increasing on $[0,\max u]$.
\end{theorem}
\begin{proof}
For the proof we can follow precisely the proof of the corresponding result in \cite{bakryqian}. We only need invariance of $\m_{\sX}$ with respect to $P_t^{\sX}$ and Corollary \ref{auchwichtig}.
\end{proof}
\begin{corollary}\label{coro}
Let $u$ and $w$ be as in the previous theorem. Then there exists a constant $c>0$ such that
\begin{align*}
\m_{\sX}(B_r(p))\leq cr^{N'}
\end{align*}
for sufficiently small $r>0$ where $p\in X$ such that $u(p)=\min u$.
\end{corollary}
\begin{proof} We know that $\min u\leq 0$. We will prove that $B_{r}(p)\subset \left\{u\leq u(p)+kr^2\right\}$ for some constant $k>0$ and $r>0$ sufficiently small. For instance, this follows exactly like in proof of Theorem 3.2 in \cite{wangxia}. 
More precisely, consider $x\in B_{r}(p)$ for $r>0$ small, and $v$ that is an eigenfunction for $\lambda_1$ of the model space such that $v(-\pi{\scriptstyle \sqrt{\frac{K}{N-1}}})=u(p)$. 
The gradient estimate implies $\lip u(x) =|\nabla u|(x)\leq \tilde{C} r^2$ for some constant $\tilde{C}>0$. By definition of the local slope we obtain
\begin{align*}
u(x)\leq u(p)+ \lip u(x)r\leq u(p)+kr^2
\end{align*}
for some constant $k>0$. 
Let $kr^2\in[\min u,-\frac{1}{2}]$.
Therefore
\begin{align*}
\m_{\sX}(B_{r}(p))\leq \int_{\left\{u\leq kr^2\right\}}ud\m_{\sX}\leq C \int_{\left\{v\leq kr^2\right\}}d\m_{\sK,\sN}
\end{align*}
where $C=\int_{\left\{u\leq 0\right\}}ud\m_{\sX}/\int_{\left\{v\leq 0\right\}}d\m_{\sK,\sN}$. For the same reason there is a constant $M>0$ such that $\left\{v\leq r^2\right\}\subset B_{Mr}(-{\textstyle\frac{\pi}{2}\sqrt{K/(N-1)}})$ in $I_{\sK/(\sN-1)}$. 
It follows that
\begin{align*}
\m_{\sX}(B_{kr}(p))\leq\m_{\sK,\sN}(B_{Kr}(-{\textstyle\frac{\pi}{2}}\sqrt{K/(N-1)}))\leq \tilde{c}r^{\sN}
\end{align*}
that implies the assertion for some constant $c>0$.
\end{proof}
\begin{theorem}
Consider $X$ and $u$ as in Theorem \ref{AAA}. Consider $v_{\sK',\sN'}$ that solves the following ordinary differential equation
\begin{align}\label{ode}
v''-\frac{K'\sin_{\sK'/(N'-1)}}{\cos_{\sK'/(\sN'-1)}}v'+\lambda_1 v=0, \ \ v(-\pi{\scriptstyle \sqrt{\frac{K'}{N'-1}}})=\min u \ \ \& \ \ v'(-\pi{\scriptstyle \sqrt{\frac{K'}{N'-1}}})=0.
\end{align}
Let $b(K',N'):=\inf\left\{x>-\pi\sqrt{\frac{K}{N-1}}: w'(x)=0\right\}$. Then $\max u \geq v_{\sK,\sN}(b(K,N))$.
\end{theorem}
\begin{proof}
Assume the contrary. If $\max u < v(b)$. Then we consider a solution $w$ of (\ref{ode}) with parameters $K'<K$ and $N'>N$. By definition $w$ is an eigenfunction of $([-\pi{\scriptstyle \sqrt{\frac{K}{N-1}}},b],\m_{\sK',N'})$ (with Neumann boundary conditions). 
Since the solution of (\ref{ode}) depends continuously on the coefficients, $w$ still satisfies $\max u<w(b)$ provided $(K',N')$ is close enough to $(K,N)$.
Then, we can apply Corollary \ref{coro}.
But on the other hand, by the Bishop-Gromov volume growth estimate
\begin{align*}
0<C\leq\frac{\m_{\sX}(B_r(p))}{r^{N}}<\frac{\m_{\sX}(B_r(p))}{r^{N'}}\rightarrow \infty \ \ \mbox{for} \ \ r\rightarrow 0.
\end{align*}
Thus we have a contradiction.
\end{proof}
\begin{proofmain}
Consider $X$ and $u$ as in Theorem \ref{AAA}. We have to check that $-\min u=\max u$. Assume $-\min u\geq \max u$. Otherwise replace $u$ by $-u$. The previous corollary tells us that 
$\max u\geq v_{\sK,\sN}(b(K,N))$ where $v_{\sK,\sN}$ is a solution of (\ref{ode}) for $K$ and $N$. 
In this case $b(K,N)=\pi{\scriptstyle \sqrt{\frac{K'}{N'-1}}}$ and $v_{\sK,\sN}=-\min u\cdot\sin_{\sK/(\sN-1)}$.
Therefore $\max u=-\min u$.\qed
\end{proofmain}

\begin{corollary}\label{bf}
Let $(X,\de_{\sX},\m_{\sX})$ be a metric measure space that satisfies the condition $RCD^*(K,N)$ for $K>0$ and $N\in (1,\infty)$.
Assume there is $u\in D^2(L^{\sX})$ such that $L^{\sX}u=-\frac{K N}{N-1}u$. Then $u$ is $K/(N-1)u$-affine. More precisely, $u\circ \gamma$ solves
\begin{eqnarray*}
&v''+ \frac{K}{N-1}|\dot{\gamma}|^2 v=0 \ \mbox{ on } \ [0,1]\\
&v(0)=u(\gamma(0))\ \& \ v(1)=u(\gamma(1))&
\end{eqnarray*}
for any geodesic $\gamma:[0,1]\rightarrow X$.
\end{corollary}
\begin{proof}
Assume $\max u=-\min u=1$ and $K=N-1$.
The maximal diameter theorem implies $X\times_{\sin_{\sK/(\sN-1)}}^{\sN}Y$ for some metric measure space $Y$. We show that $u(r,x)=u_1(r)\otimes 1$ for some measurable function $u_1:[0,\pi\sqrt{\frac{K}{N-1}}]\rightarrow \mathbb{R}$. 
Then 
\begin{align*}
-Nu_1\otimes 1=L^{\sX}u_1\otimes 1=L^{[0,\pi],\sin}u_1\otimes 1=\left(\frac{d^2}{dr^2}u_1-N\frac{\cos}{\sin}\frac{d}{dr}u_1\right)\otimes 1.
\end{align*}
Therefore, $u_1$ is an eigenfunction of $L^{[0,\pi],\sin}$ and from Theorem \ref{crucial} follows $u_1$ satisfies the statement. But because of the suspension structure of $X$ and since $u=u_1\otimes 1$ it holds for $u$ as well.
\smallskip\\
Let $x$ and $y$ be the maximum and minimum point of $u$ respectively, and let us consider a level set $\left\{x\in X: u(x)=L\right\}=\mathcal{L}$ of $u$. 
Since $X$ is a spherical suspension, for any $z\in \mathcal{L}$ there is exactly one geodesic $\gamma$ that connects $x$ and $y$ sucht that $z=\gamma(t)$ for some $t\in (0,1)$. The gradient comparison result for $u$ again yields
\begin{align}\label{geraeusch}
\arccos u(x)-L\leq \de_{\sX}(x,z)\ \ \& \ \ L-\arccos u(y)\leq \de_{\sX}(z,y)
\end{align}
and 
\begin{align}
\pi= \arccos\circ u(x)-\arccos\circ u(y)\leq \de_{\sX}(x,\gamma(t))+\de_{\sX}(\gamma(t),y)=\pi. 
\end{align}
Hence, we have equality in (\ref{geraeusch}), and since $\mathcal{L}=\partial B_{\pi-L}(x)=B_{\pi+L}(y)$ because of the suspension structure, $u$ doesn't depend on the second variable
\end{proof}

\begin{corollary}\label{brbrbrbr}
Let $(X,\de_{\sX},\m_{\sX})$ be a metric measure space that satisfies the condition $RCD^*(K,N)$ for $K\geq 0$ and $N\geq 1$. If $N>1$, we assume $K>0$, and if $N=1$ we assume $K=0$ and $\diam_{\sX}\leq \pi$.
There is $u\in D^2(L^{\sX})$ such that
\begin{itemize}
 \item[(i)] $L^{\sX}u=-\textstyle{\frac{K N}{N-1}}u \ \mbox{ if }\ N>1,$
 \smallskip
 \item[(ii)] $L^{\sX}u=-Nu \ \mbox{ otherwise }.$
\end{itemize}
Then, $X\!=\![0,\pi]\times_{\sin_{K/(N-1)}}^{N-1}X'$ for a metric measure space $X'$, and $u=c\cdot\cos_{{\scriptscriptstyle {K}/{N-1}}}$.
\end{corollary}

\section{Higher eigenvalue rigidity}\label{higher}
\begin{proofhigher}
\textbf{1.} First, let $k\leq N$.
We introduce some notations. We call the warped product $\mathbb{S}_+^k(1)\times ^{N-k}_{f}Z$ $k$-multi-suspension, 
$r\times Z$ fiber at $r\in \mathbb{S}_+^k(1)\backslash \partial \mathbb{S}_+^k(1)$ and $\mathbb{S}_+^k(1)\times p$ $k$-base at $p$.
$\sin\circ\de_{\partial \mathbb{S}_+^k(1)}(r)=:f(r)$ is the sin of the distance of $r$ from the boundary. 
$f$ can be constructed inductively for any $k$ as follows. One knows $\mathbb{S}_+^k(1)=\mathbb{S}_+^{k-1}(1)\times_{g}^1[0,\pi]$ where 
$g=\sin\circ\de_{\partial \mathbb{S}_{\scriptscriptstyle{+}}^{k-1}(1)}(r)$. Then $f=\sin\circ g\otimes \sin$.
\smallskip\\
We assume without restriction $N>1$. Because of the Lichnerowicz estimate for metric measure spaces satisfying $RCD^*(N-1,N)$, $\lambda_{k}=N$ implies $\lambda_1,\dots, \lambda_{k-1}=N$. In particular, there is a set of linearly independent eigenfunctions $u_1,\dots, u_k\in D(L^{\sX})$.
\smallskip
\\
\textbf{2.} By Obata's theorem $\lambda_1=N$ implies that there exists $x_1,y_1$ in $X$ such that $\de_{\sX}(x_1,y_1)=\pi$. Therefore, $X$ splits a spherical $(N-1)$-suspension $X=[0,\pi]\times^{N-1}_{\sin}X'$ 
for some metric measure space $X'$ that satisfies $RCD^*(N-2,N-1)$. By Corollary \ref{brbrbrbr} we know that the corresponding eigenfunction $u$ does not depend on $X'$, it is essentially an eigenfunction of 
$([0,\pi],\sin^{\sN-1})$ with eigenvalue $N$ and $u=c\cdot\cos$. 
\smallskip
\\
\textbf{3.} From $\lambda_2=N$ follows the same for points $x_2,y_2$ and $(x_2,y_2)\neq (x_1,y_1)$.  
The latter holds since otherwise $u_2=\cos\otimes c'$ for some constant $c'>0$ on $X=[0,\pi]\times^{N-1}_{\sin}X'$ by the previous step what contradicts linear independency of $u_1$ and $u_2$.
Since $X$ has two different suspension structures, there is a geodesic circle in $X$ that intersects $X'$ twice at points $x',y'\in X'$ such that $\de_{X}(x,y)=\de_{\sX}(x',y')=\pi$ ($X'$ embeds into $X$). Therefore, $X'$ splits, too, 
and we obtain that
$$X=[0,\pi]\times^{N-1}_{\sin}X'=[0,\pi]\times^{N-1}_{\sin}\left([0,\pi]\times^{N-2}_{\sin}X''\right).$$
Since the warped procut construction is associative (see the proof of Corollary 3.19 in \cite{ketterer}), we get
$$X=\mathbb{S}_+^2(1)\times^{N-2}_{\sin\circ \de_{\partial \mathbb{S}_{\scriptscriptstyle{+}}^{2}(1)}}X''$$ where $\mathbb{S}^2_+(1)$ is the $2$-dimensional upper hemisphere of the standard sphere with radius $1$.
\smallskip
\\
\textbf{4.} We continue by induction. $\lambda_3=N$. Again, there are points $(x_3,y_3)\in X$ such that $\de_{\sX}(x_3,y_3)=\pi$ and there is decomposition of $X$ with respect 
to these points $[0,\pi]\times ^{N-1}_{\sin}Y$ such that $u_3=\cos\otimes r$ for some constant $r>0$.
Consider the decomposition $$X=\mathbb{S}_+^2(1)\times^{N-2}_{\sin\circ\de_{\partial \mathbb{S}_+^2(1)}}X''.$$ Then, a geodesic $\gamma\sim[0,\pi]\times p$ ($p\in Y$) that connects
$x_3$ and $y_3$ is not contained in some $2$-base $\mathbb{S}_+^2(1)\times{q}$ for $q\in X''$. Otherwise, the spherical splitting with respect to $x_3,y_3$ would not affect $X''$ and 
$u_3$ would not depend on the $X''$-variable. More precisely, $u_3$ would be an eigenfunction on $$(\mathbb{S}_+^2(1),\sin^{\sN-2}\circ\de_{\partial\mathbb{S}_+^2(1)})$$ with eigenvalue $N$. 
However, the eigenspace for $\lambda_1=N$ of $(\mathbb{S}_+^2(1),\sin^{\sN-2}\circ\de_{\partial\mathbb{S}_+^{2}(1)})$ is
just $2$-dimensional. It follows $u_2=a\cdot u_1+b\cdot u_2$, and we obtain a contradiction.
Hence, $x_3=(r,p)$ and $y_2=(s,q)$ in $\mathbb{S}_+^2(1)\times X''$ for $q\neq p\in X''$. We can repeat step \textbf{2.} and obtain $$X=\mathbb{S}^3_+(1)\times^{N-3}_{\sin\circ\de_{\partial \mathbb{S}^3_+(1)}}X'''.$$
\smallskip
\\
\textbf{5.}
The decomposition continues until $$X=\mathbb{S}^{n-1}_+(1)\times_{\sin \circ\de_{\partial\mathbb{S}^{n-1}_+(1)}}^{N-n} Z$$ for some metric measure space $Z$ such that $N-n<1$. $Z$ is either a 
single point, or $Z$ consists of exactly two points a distance $\pi$, and it follows $N=n=k$. In the latter case we also obtain $\lambda_{k+1}=N$ since $X$ is the standard sphere.
\smallskip
\\
\textbf{6.}
If $k>N$, then we can repeat the previous steps for $l\in\mathbb{N}$ that is the biggest integer smaller than $N$. Hence, $X$ is either the upper hemi-sphere or the $l$-dimensional standard sphere. 
In the latter case it follows $N=l$.
But since $\lambda_{l+1}=N$, $X$ has to be the sphere and $l+1=N+1=k$. \qed\end{proofhigher}

\section{Hessian lower bounds and convexity}\label{hes}
The main result of this section is Theorem \ref{convexity}. Since the space $\mathbb{D}^{\sX}_{\infty}$ is too restrictive, in the following definition we will extend the class of functions that admit a Hessian. 
\begin{definition}
Let $(X,\de_{\sX},\m_{\sX})$ be a metric measure space satisfying $RCD(\kappa,N)$. 
We say that $V\in W^{1,2}(X)$ admits a \textit{Hessian} if $|\nabla V|^2\in L^{\infty}(\m_{\sX})$ and for any $u\in \mathbb{D}_{\infty}^{\sX}$ we have $\langle \nabla V,\nabla u\rangle \in W^{1,2}(X)$. 
In this case $H[V]$ is defined as
\begin{align*}
H[V](u)=\frac{1}{2}\langle\nabla V,\nabla |\nabla u|^2\rangle-Du(\nabla \langle\nabla V,\nabla u\rangle).
\end{align*}
In particular, any $V\in \mathbb{D}^{\sX}_{\infty}$ admits a Hessian, and the definitions coincide.
We say $H[u]\geq K$ for some constant $K\in \mathbb{R}$ if
\begin{align*}
\int H[u](f,g)\phi d\m_{\sX}\geq K\int \langle \nabla f,\nabla g\rangle\phi d\m_{\sX}
\end{align*}
for any $f,g\in \mathbb{D}_{\infty}^{\sX}$ and for any $\phi\in D^{\infty}_+(L^{\sX})$.
\end{definition}
\begin{theorem}\label{convexity}
Let $(X,\de_{\sX},\m_{\sX})$ be a metric measure space that satisfies $RCD(\kappa,N)$, and let $V\in W^{1,2}(X)$ such that $\delta\leq V\leq \delta^{-1}$ for some $\delta>0$ and such that $V$ admits a Hessian.
Then the following statements are equivalent:
\begin{itemize}
\item[(i)]$V$ is continuous, and $H[V]\geq K$,
\smallskip
\item[(ii)]$V$ is $K$-convex.
\end{itemize}
\end{theorem}

\begin{proof}``$\Longrightarrow$'': Since $RCD^*(\kappa,N)$ implies $RCD(\kappa,\infty)$, the Cheeger enery $\ChX$ satisfies $BE(\kappa,\infty)$ \cite{agsbakryemery}:
\begin{align*}
\int_X|\nabla u|^2 L^{\sX}\phi\ d\m_{\sX}\geq \int_X \langle \nabla u, \nabla L^{\sX}u\rangle \phi\ d\m_{\sX}+\kappa\int_X|\nabla u|^2d\m_{\sX}
\end{align*}
for any pair $(u,\phi)$ with $\phi\geq 0$ and $\phi\in D^{\infty}(L^{\sX})$ and $u\in D(\Gamma_2^{\sX})$. 
On the other hand $H[V]\geq K$ implies that
\begin{align*}
\int_X\langle \nabla \langle\nabla V,\nabla g\rangle,\nabla g\rangle\varphi d\m_{\sX}- \frac{1}{2}\int_{X}\langle \nabla \langle \nabla g,\nabla g\rangle,\nabla V\rangle\varphi d\m_{\sX}\geq K\int_{\sX}|\nabla g|^2\varphi d\m_{\sX}
\end{align*}
for any $g\in \mathbb{D}_{\infty}$ and $\varphi\in D^{\infty}_+(L^{\sX})$. $V$ is Lipschitz since $|\nabla V|\in L^{\infty}(X)$ and $X$ satisfies $RCD(\kappa,\infty)$ \cite{agsriemannian}.
\smallskip
\paragraph{\textbf{1.}} We will show that the transformed space $(X,\de_{\sX},\m_{\sX,\sV})$ with $d\m_{\sX,\sV}:=e^{-V}d\m_{\sX}$ satisfies the condition $RCD(\kappa+K,\infty)$.
\smallskip

Since $\delta \leq e^{-\sV}\leq \delta^{-1}$ on $X$,
it follows that $L^p(\m_{\sX})=L^p(\m_{\sX,\sV})$ for any $p\in [1,\infty]$ and $W^{1,2}(\m_{\sX})=W^{1,2}(\m_{\sX,\sV})$.
 If we choose $v\in W^{1,2}_0(\m_{\sX,\sV})$ that is Lipschitz continuous, then $\bar{v}:=ve^{V}\in W^{1,2}_0(\m_{\sX,\sV})$ and it is Lipschitz continuous as well. Then, if we consider $u\in D(L^{\sX,\sV})$, we have
\begin{align*}
\int_X \bar{v}L^{\sX,\sV}u d\m_{\sX,\sV}&=-\int_X\langle\nabla \bar{v},\nabla u\rangle e^{-V}d\m_{\sX}\\
&=-\int_X\left[\langle \nabla \bar{v}e^{-V},\nabla u\rangle -\langle \nabla V,\nabla u\rangle \bar{v} e^{-V}\right]d\m_{\sX}
\\
&=-\int_X\langle \nabla \bar{v}e^{-V},\nabla u\rangle d\m_{\sX}+\int_X\langle \nabla V,\nabla u\rangle \bar{v} e^{-V}d\m_{\sX}
\end{align*}
Therefore, %
the measure valued Laplacian of $u$ with respect to $\m_{\sX}$ has density 
$$L^{\sX}u=L^{\sX,\sV}u-\langle \nabla V,\nabla u\rangle\in L^2(\m_{\sX}).$$
Similar, one can see that, if $u\in D^2(L^{\sX})$, the measure valued Laplacian of $u$ with respect $\m_{\sX,\sV}$ has density $L^{\sX}u+\langle \nabla V,\nabla u\rangle\in L^2(\m_{\sX})$.
Hence, $D^2(L^{\sX})=D^2(L^{\sX,\sV})$.
Consider 
\begin{align*}
\Xi:=\bigcup_{t>0}P_t^{\sX}\left[D(L^{\sX})\right] \ \ \& \ \ \Xi_+:=\bigcup_{t>0}P_t^{\sX}\left[D_+^{\infty}(L^{\sX})\right]
\end{align*}
The curvature-dimension condition for $(X,\de_{\sX},\m_{\sX})$ implies that $\Xi\subset \mathbb{D}_{\infty}^{\sX}$ and $\Xi_+\subset D^{\infty}_+(L^{\sX})\cap\mathbb{D}^{\sX}_{\infty}$. 
In particular, $u, |\nabla u|_{\sX}, L^{\sX}u\in L^{\infty}(\m_{\sX})$ if $u\in \Xi$ or $u\in \Xi_+$, and $|\nabla u|^2\in W^{1,2}(\m_{\sX})=W^{1,2}(\m_{\sX,\sV})$. 
Therefore, if we consider $u\in \Xi$, we have $u\in D^2(L^{\sX,\sV})$ and $L^{\sX,\sV}\in W^{1,2}(\m_{\sX,\sV})$.
Hence, we can compute $\Gamma^{\sX,\sV}_2(u,\phi)$ for $u\in \Xi$ and $\phi\in \Xi_+$.
\begin{align*}
\Gamma_2^{\sX,\sV}(u;\phi)=\frac{1}{2}\underbrace{\int_X|\nabla u|^2L^{\sX, V}\phi d\m_{\sX,\sV}}_{=:(I)} - \underbrace{\int_X\langle\nabla u, \nabla L^{\sX,\sV} u\rangle \phi d\m_{\sX,\sV}}_{=:(II)}
\end{align*}
Since $|\nabla u|^2\in W^{1,2}(\m_{\sX,\sV})$, we can continue as follows.
\begin{align*}
(I)&=-\int_X\langle\nabla|\nabla u|^2,\nabla \phi\rangle e^{-V}d\m_{\sX}\\
&=-\int_X\langle\nabla|\nabla u|^2,\nabla \phi e^{-V}\rangle d\m_{\sX}+\int_{X}\langle\nabla |\nabla u|^2,\nabla e^{-V}\rangle\phi d\m_{\sX}\\
(II)&=\int_X\langle\nabla u, \nabla L^{\sX} u\rangle \phi e^{-V}d\m_{\sX}-\int_X \langle \nabla \langle\nabla V,\nabla u\rangle,\nabla u\rangle\phi e^{-V}d\m_{\sX}
\end{align*}
Since $\phi e^{-V}\in W^{1,2}(\m_{\sX})$ and $-\int\langle\nabla|\nabla u|^2,\nabla \phi e^{-V}\rangle d\m_{\sX}=\int \phi e^{-V} d{\boldsymbol L}^{\sX}|\nabla u|^2$, Lemma \ref{sav} implies
\begin{align*}
\frac{1}{2}(I)-(II)&\geq \kappa\int_X|\nabla u|^2\phi d\m_{\sX,\sV}+\int_XH[V](u,u)\phi d\m_{\sX,V}\\
&\geq (\kappa+K)\int_X|\nabla u|^2d\m_{\sX}.
\end{align*}
We extend the previous estimate from $\Xi$ to $D(\Gamma_2^{\sX,\sV})$. Let $u\in D(\Gamma_2^{\sX,\sV})\subset D(L^{\sX,\sV})$. Then $P_t^{\sX}u\in \Xi$. 
$P_t^{\sX}u\rightarrow u$ in $D(L^{\sX,V})$ if $t\rightarrow 0$. To see that we observe
\begin{align*}
\left\|L^{\sX,\sV}P_t^{\sX}u-L^{\sX,\sV}u\right\|_{L^{2}}\leq \underbrace{\left\|L^{\sX}\left(P_t^{\sX}u-u\right)\right\|_{L^2}}_{\rightarrow 0}+\underbrace{\left\||\nabla V|\right\|_{L^2}\left\|P_t|\nabla u|-|\nabla u|\right\|_{L^2}}_{\rightarrow 0}.
\end{align*}
Hence, we can take the limit in $\Gamma_2^{\sX,\sV}(P^{\sX}_tu;P^{\sX}_s\phi)\geq 0$ if $t,s\rightarrow 0$ (first $t$ then $s$), and we obtain $\Gamma_2^{\sX,\sV}(u;\phi)\geq 0$
for $u\in D(\Gamma_2^{\sX,\sV})$ and $\phi \in D^{\infty}(L^{\sX,\sV})$. For instance, this works precisely like in paragraph 3. and 4. of the proof of Theorem 3.23 in \cite{ketterer2}.
Hence, $\Ch^{\sV,\sX}$ satisfies $BE(\kappa+K,\infty)$ and by the equivalence result from \cite{agsbakryemery} 
$(X,\de_{\sX},e^{-V}d\m_{\sX})$ satisfies $RCD(\kappa+K,\infty)$. 
\medskip
\paragraph{\textbf{2.}}
In the last step we use the proof of Sturm's gradient flow result from \cite{sturmverynew}. Roughly speaking, a time shift transformation applied to $\ChX$ makes the diffusion part of the entropy gradient flow of $(X,\de_{\sX},\m_{\sX,\sV})$ disappear. 
Let us briefly sketch the main idea.
Let $\alpha>0$. Consider $X^{\alpha}=(X,\alpha^{-1} \de_{\sX})$ and $V^{\alpha}=V/\alpha^2$. It is easy to check that 
$|\nabla u|_{X^{\alpha}}=\alpha|\nabla u|_{X}$. Then $H[{V^{\alpha}}]\geq K$ with respect to $|\nabla \cdot|_{X^{\alpha}}$ (one can check that $H_{\sX}[V](u)=\alpha^{-2}H_{\sX^{\alpha}}[V^{\alpha}](u)$). $(X^{\alpha},\m_{\sX})$ satisfies $RCD(\alpha^{2}\kappa,\infty)$. By the previous part of the proof, it 
follows that $(X^{\alpha},\m_{\sX^{\alpha},\sV^{\alpha}})$ satisfies $RCD(\alpha^{2}\kappa+K,\infty)$. Hence, there exists a gradient flow curve $\mu_t$ of $\mbox{Ent}_{\m_{\sX^{\alpha},\sV^{\alpha}}}$ that satisfies
\begin{align*}
\frac{1}{2\alpha^2}\frac{d}{dt}\de_{W}(\mu_t,\nu)^2+\frac{\alpha^2\kappa+K}{\alpha^2 2}\de_{W}(\mu_t,\nu)^2\leq \Ent_{\m_{\sX^{\alpha},\sV^{\alpha}}}(\nu)-\Ent_{\m_{\sX^{\alpha},\sV^{\alpha}}}(\mu_t).
\end{align*}
One can easily check that 
$$
\Ent_{\m_{\sX^{\alpha},\sV^{\alpha}}}(\mu)=\Ent_{\m_{\sX^{\alpha}}}(\mu)+\int_XV^{\alpha}d\mu=\Ent_{\m_{\sX}}(\mu)+\int_X\alpha^{-2}V d\mu.
$$
If we multiply by $\alpha^2$ and let $\alpha\rightarrow 0$, a subsequence of $(\mu^{(\alpha)}_t)_{\alpha}$ converges to a evi$_{K}$ gradient flow curve of $\int Vd\mu$ (see \cite{sturmverynew} for details).
By an estimate for the moments (see Lemma 3 in \cite{sturmverynew}) the evolution is non-diffusive and   
the contraction property guarantees uniqueness for any starting point. Hence, this yields unique evi$_{K}$ gradient flow curves in the underlying space for $\m_{\sX}$-a.e. starting point. 
And by the contraction estimate again it is true for every point. 
It follows that $V$ admits unique evi$_K$ gradient flow curves for any starting point $p\in X$ that implies $V$ is $K$-convex. For details we refer to \cite{sturmverynew}.
\medskip\\
``$\Longleftarrow$'': $(X^{\alpha}, e^{-V^{\alpha}}\m_{\sX})$ satisfies $RCD(\alpha^{2}\kappa+K,\infty)$ \cite{agsriemannian}. Hence, its Cheeger energy satisfies $BE(\alpha^{2}\kappa+K,\infty)$ \cite{agsbakryemery}.
Since $V\in \mathbb{D}^{\sX}_{\infty}$, we can compute 
\begin{align*}
\Gamma_2^{\sX^{\alpha},\m_{\sX^{\alpha},\sV^{\alpha}}}(u;\phi)=
\Gamma_2^{\sX}(u;\phi e^{-V^{\alpha}})+\int_X\underbrace{H[V^{\alpha}](u)}_{=\alpha^2H[V](u)}\phi d\m_{\sX,\sV^{\alpha}}\geq \dots 
\end{align*}
for any $u\in\mathbb{D}_{\infty}^{\sX}=\mathbb{D}_{\infty}^{\sX^{\alpha}}$ and for 
any suitable $\phi>0$.
The result follows if multiply the previous inequality by $\alpha^{-2}$ and let $\alpha\rightarrow 0$.
\end{proof}

\section{Final remarks}
We want to make a few additional comments on the non-Riemannian case. The Lichnerowicz spectral gap estimate also holds in the case when $(X,\de_{\sX},\m_{\sX})$ just satisfies $CD(K,N)$ for $K>0$ and $N\geq 1$.
One can ask if we obtain a similar rigidity results in this situation. The failure of a metric splitting theorem for non-Riemannian $CD$-spaces indicates 
that one can not hope for a metric Obata theorem for non-Riemannian spaces. 
But a topological rigidity result might be true. Indeed, for weighted Finsler manifolds that satisfy a curvature-dimension condition $CD(K,N)$ for $K>0$ and $N\geq 1$, the following theorem is an 
easy consequence of results by Ohta \cite{ohtpro} and Wang/Xia \cite{wangxia}.
\begin{theorem}
Let $(X,\mathcal{F}_{\sX},\m_{\sX})$ be a weighted Finsler manifold that satisfies the condition $CD(K,N)$ for $K>0$ and $N>1$. 
Assume there is $u\in C^{\infty}(X)$ such that
$$L^{\sX}u=-\frac{K N}{N-1}u$$ 
where $L^{\sX}$ is the Finsler Laplacian with respect to $\m_{\sX}$.
Then, there exists a Polish space $(X',\m_{\sX'})$ such that the measure space $(X,\m_{\sX})$ is isomorphic to a topological suspension over $X'$.
\end{theorem}
\begin{proof}
Theorem 3.1 in \cite{wangxia} implies that $\lip (\sin^{-1}(u))\leq 1$. Hence, there are points $x,y\in X$ such that $\de_{\mathcal{F}_X}(x,y)=\pi$, and we can apply Theorem 5.5 from \cite{ohtpro}.
\end{proof}
\begin{remark}
The previous Theorem suggests that a topological eigenvalue rigidity result may also hold in a more general class of metric measure spaces. 
\end{remark}

\small{
\bibliography{new}

\bibliographystyle{amsalpha}
}
\end{document}